\documentclass[11pt,reqno]{amsart}

\usepackage{graphicx}
\usepackage{color}
\usepackage{fullpage} 

\usepackage[all]{xy}
\usepackage{tikz}
\usetikzlibrary{decorations.pathreplacing}
\usetikzlibrary{arrows}

\usepackage{calrsfs}
\usepackage[T1]{fontenc} 
\usepackage{textcomp}
\usepackage{times}
\usepackage[scaled=0.92]{helvet} 

\usepackage{etoolbox}
\patchcmd{\section}{\scshape}{\bfseries}{}{}
\makeatletter
\renewcommand{\@secnumfont}{\bfseries}
\makeatother

\newtheorem{introtheorem}{Theorem}

\theoremstyle{plain}
\newtheorem{theorem}{Theorem}[section]
\newtheorem{proposition}[theorem]{Proposition}
\newtheorem{lemma}[theorem]{Lemma}
\newtheorem{kellylemma}[theorem]{Kelly's Lemma}
\newtheorem{corollary}[theorem]{Corollary}
\newtheorem{conjecture}[theorem]{Conjecture}
\theoremstyle{definition}
\newtheorem{definition}[theorem]{Definition}
\newtheorem{notation}[theorem]{Notation}

\newtheorem{remark}[theorem]{Remark}
\newtheorem{question}[theorem]{Question}

\usepackage{amsfonts}
\usepackage{amssymb}

\makeatletter
\newcommand{\oset}[3][0ex]{%
  \mathrel{\mathop{#3}\limits^{
    \vbox to#1{\kern-\ex@
    \hbox{$\scriptstyle#2$}\vss}}}}
\makeatother
\usepackage{stmaryrd}
\def\la#1{\oset{\vspace*{-5mm}\shortleftarrow}{#1}}
\def\ra#1{\oset{\shortrightarrow}{#1}}

\renewcommand{\bar}{\overline}

\DeclareMathOperator{\Z}{\mathbf{Z}}

\DeclareMathOperator{\R}{\mathbf{R}}
\DeclareMathOperator{\C}{\mathbf{C}}
\DeclareMathOperator{\E}{\mathbf{E}}

\DeclareMathOperator{\PF}{\mathrm{PF}}

\begin{document}

\date{\today\ (version 2.0)} 
\title[Edge reconstruction of the Ihara zeta function]{Edge reconstruction\\[1mm]  of the Ihara zeta function}
\author[G.~Cornelissen]{Gunther Cornelissen}
\address{\normalfont Mathematisch Instituut, Universiteit Utrecht, Postbus 80.010, 3508 TA Utrecht, Nederland}
\email{g.cornelissen@uu.nl}
\author[J.~Kool]{Janne Kool\\ (with an appendix by Daniel McDonald)}
\address{\normalfont  Max-Planck-Institut f\"ur Mathematik, Postfach 7280 53072 Bonn, Deutschland}
\curraddr{\normalfont Kognitive Systemer, DTU Compute, Dansk Tekniske Universitet, B321, DK-2800 Lyngby}
\email{jankoo@dtu.dk}
\thanks{Part of this work was done while the first author visited Warwick (made possible by Richard Sharp) and the Hausdorff Institute in Bonn, and while the second author visited the Max-Planck-Institute in Bonn. We thank Tom Kempton and Matilde Marcolli for stimulating discussions and Toshikazu Sunada for useful comments. The proof of Lemma \ref{pol} is a simplification of our original argument, due to Merlijn Staps.}

\subjclass[2010]{05C50, 05C38, 11M36, 37F35, 53C24}
\keywords{\normalfont Graph, edge reconstruction conjecture, Ihara zeta function, non-backtracking walks}

\begin{abstract} \noindent We show that if a graph $G$ has average degree $\bar d \geq 4$, then the Ihara zeta function of $G$ is edge-reconstructible. We prove some general spectral properties of the edge adjacency operator $T$: it is symmetric for an indefinite form and has a ``large'' semi-simple part (but it can fail to be semi-simple in general).  We prove that this implies that if $\bar d>4$, one can reconstruct the number of non-backtracking (closed or not) walks through a given edge, the Perron-Frobenius eigenvector of $T$ (modulo a natural symmetry), as well as the closed walks that pass through a given edge in both directions at least once.

The appendix by Daniel MacDonald established the analogue for multigraphs of some basic results in reconstruction theory of simple graphs that are used in the main text. 
\end{abstract}

\maketitle


\section*{Introduction} 
Let $G=(V,E)$ denote a graph with vertex set $V$ and edge set $E$, consisting of unordered pairs of elements of $V$. The \emph{edge deck} $\mathcal{D}^e(G)$ of $G$ is the multi-set of all edge-deleted subgraphs of $G$, as unlabelled graphs. Harary \cite{Harary} conjectured in 1964 that graphs on at least four edges are edge-reconstructible, i.e., determined up to isomorphism by their edge deck. 
This so-called \emph{edge reconstruction conjecture} is the analogue for edges of the famous vertex reconstruction conjecture of Kelly and Ulam that every graph on at least three vertices is determined by its (similarly defined) vertex deck (compare \cite{Bondy}). Many invariants of graphs were shown to be reconstructible from the vertex and/or edge deck. From the large literature on the subject, we quote the following three sources that are most relevant in the context of our results: (a) vertex-reconstruction of the characteristic polynomial of the vertex adjacency matrix by Tutte \cite{Tutte}; (b) vertex-reconstruction of the number of (possibly backtracking) walks of given length through a given vertex $v \in V$ (which one can specify without knowing the graph $G$ by pointing to the element $G-v$ of the vertex deck) by Godsil and McKay \cite{Godsil}; (c) edge reconstruction for graphs with average degree $\bar d \geq  2 \log_2 |V|$ by Vladim\'{\i}r M\"uller \cite{Muller}, improving upon a method of Lov\'asz \cite{Lovasz}. 

Following the discussion by McDonald in the appendix to the current paper, the edge reconstruction conjecture should also hold for multigraphs in the formulation of Conjecture \ref{ercm}. Since disconnected (multi)graphs are reconstructible (see (\cite{Bondy},Corollary 6.14(b)) and \ref{appcor}(\ref{c4})), we may assume that $G$ is connected. An edge with equal ends is called a loop. The degree of a vertex is the number of edges to which it belongs, where, as usual, a loop is counted twice. The average degree $\bar d$ of $G$ then equals $$\bar d = \frac{1}{|V|} \sum_{v \in V} \deg v = 2 \frac{|E|}{|V|}.$$ A degree-one vertex is called an end-vertex. All results in this paper hold for \emph{connected finite undirected multigraphs without end-vertices}, and from now on we will use the word ``graph'' for such multigraphs.

If $e=\{v_1,v_2\} \in E$, we denote by $\ra{e}\ =(v_1,v_2)$ the edge $e$ with a chosen orientation, and by $\la{e}\ =(v_2,v_1)$ the same edge with the inverse orientation to that of $\ra{e}$. Let $o(\ra{e})=v_1$ denote the origin of $\ra{e}$ and $t(\ra{e})=v_2$ its end point. If there are multiple edges between $v_1,v_2$ then we will label them $e_i=(v_1,v_2)_i$. A non-backtracking edge walk of length $n$ is a sequence $e_1e_2....e_n$ of edges such that $t(e_i)=o(e_{i+1})$, but $\la{e_{i+1}}\not=\ra{e_{i}}$. We call it \emph{tailless} if $\la{e_n} \neq \ra{e_1}$. Just like walks in the graph can be studied using the adjacency matrix, non-backtracking walks are captured by the  \emph{edge adjacency matrix} $T=T_G$ studied by Sunada \cite{Sunada1}, Hashimoto \cite{Hashimoto} and Bass \cite{Bass}. Letting $\E$ denote the set of oriented edges of $G$ for any possible choice of orientation, so $|\E|=2|E|$, $T$ is defined to be the $2|E|\times2|E|$ matrix, in which the rows and columns are indexed by $\E$, and 
$$T_{\ra{e_1},\ra{e_2}}= \left\{ \begin{array}{l} 1 \mbox{ if }t(\ra{e_1})=o(\ra{e_2}) \mbox{ but }\ra{e_2}\ \neq\ \la{e_1}; \\ 0 \mbox{ otherwise}. \end{array} \right.$$  If $r \in \Z_{\geq 1}$, the entry $(T^r)_{\ra{e_1},\ra{e_2}}$ is the number of non-backtracking walks of length $r$ on $G$ that start in the direction of $\ra{e_1}$ and end in the direction of $\ra{e_2}$. As for the usual adjacency matrix, graphs can have the same eigenvalues for $T$ without being isomorphic (\cite{Terras}, Chapter 21). 

We will denote the unit square matrix of size $n \times n$ by $1_n$ or simply $1$ if no confusion can arise. 
The matrix $T$ is related to the \emph{Ihara zeta function} $\zeta_G$ of $G$  \cite{Ihara}, defined as the following analogue of the Selberg zeta function from differential geometry (cf.\ \cite{Terras}, Part I): \begin{equation} \label{sel1} \zeta_G(u) := \prod_{p} (1-u^{\ell(p)})^{-1}, \end{equation} where the product runs over classes of non-backtracking 
tailless closed oriented \emph{prime} walks $p$ in $G$ of length $\ell(p)$,``class'' refers to not having a distinguished starting point, and ``prime'' refers to not being a multiple of another walk. The function $\zeta_G(u)$ is a formal power series in $u$, but it is also convergent as a function of the complex variable $u$ for $|u|$ sufficiently small. We have an identity (\cite{Bass}, II.3.3) \begin{equation} \label{sel2} \zeta^{-1}_G(u)=\det(1-Tu)=u^{2|E|} \det(u^{-1}-T),\end{equation} showing that $\zeta_G$ has an analytic continuation to the entire complex plane as a rational function with finitely many poles. If one so wishes, one may take Equation (\ref{sel2}) as a definition of $\zeta_G$; in this paper, the original definition as in Equation (\ref{sel1}) will play no role. 

In the case considered in the theorem below, the matrix $T$ has a unique maximal real positive eigenvalue called the \emph{Perron-Frobenius} eigenvalue. The corresponding eigenvector $ \mathbf{p} \in \R^{\E}$ such that $\sum_{e \in E} \mathbf{p}_{\ra{e}} \mathbf{p}_{\la{e}} = 1/2$ is called the \emph{normalized Perron-Frobenius eigenvector}.

We will prove the following:

\begin{introtheorem} \label{main} Let $G$ denote a graph of average degree $\bar d$. The following are edge-reconstructible:
\begin{enumerate}
\item[\textup{(i)}] If $\bar d \geq 4$, the Ihara zeta function $\zeta_G$ of $G$, i.e., the spectrum of the edge-adjacency matrix $T$; in particular, the Perron-Frobenius eigenvalue $\lambda_{\PF}$ of $T$;
\item[\textup{(ii)}] If $\bar d \geq 4$, the number $N_r$ of  non-backtracking closed walks on $G$ of given length $r$;
\item[\textup{(iii)}] If $\bar d > 4$, the functions
\begin{enumerate}
\item[\textup{(a)}] $N_r \colon \mathcal{D}^e(G) \rightarrow \Z$ that associates to an element $G-e$ of the edge deck of $G$ the number $N_r(e)$ of  non-backtracking closed
walks on $G$ of given length $r$ passing through $e$;
\item[\textup{(b)}] $M_r \colon \mathcal{D}^e(G) \rightarrow \Z$ that associates to an element $G-e$ of the edge deck of $G$ the number $M_r(e)$ of  non-backtracking (not necessarily closed)
walks on $G$ of given length $r$ starting at $e$ (in any direction);
\end{enumerate}
\item[\textup{(iv)}] If $\bar d > 4$, the function $\mathcal{D}^e(G) \rightarrow \binom{\R}{2}$ (where $\binom{\R}{2}$ is the set of unordered pairs of real numbers) that associates to an element $G-e$ of the edge deck the unordered pair $\{\mathbf{p}_{\ra{e}},\mathbf{p}_{\la{e}}\}$ of entries of the normalized Perron-Frobenius eigenvector $\mathbf{p}$ of $T$;

\item[\textup{(v)}]  If $\bar d> 4$, the function $F_r \colon \mathcal{D}^e(G) \rightarrow \Z$ that associates to an element $G-e$ of the edge deck $\mathcal{D}^e(G)$ of $G$ the number of non-backtracking closed walks on $G$ of given length $r$ that pass through $e$ in both directions at least once.
\end{enumerate}
Furthermore, if $G$ is bipartite, then \textup{(iii)-(v)} also hold for $\bar d = 4$.
\end{introtheorem}

Statements (iii)-(v) in the theorem make sense, since if $G-e \cong G-e'$, the functions turn out to have the same value at $e$ and $e'$ (cf.\ Remark \ref{mn}). 

We indicate briefly how to prove these results. Deleting an edge from the graph corresponds to deleting \emph{two} rows and columns from the matrix $T$, namely, those corresponding to the two possible orientations of the edge. The proof of (i) starts with a lemma on the combinatorial reconstruction of the top half of the coefficients of $\det(\lambda-T)$ from {second} derivatives of $2 \times 2$-minors of $T$ (Section \ref{poly}). The next step in the proof is to exploit certain relations between the coefficients in $\det(\lambda-T)$ which arise from a formula of Bass that relates $\det(\lambda-T)$ to a polynomial of degree $2|V|$---there are enough relations to reconstruct all coefficients if the stated condition on the average degree holds (Section \ref{Sbass}; in a sense, this is an analogue of the ``functional equation'' for the Ihara zeta function of a \emph{regular} graph). Part (ii) follows by expressing the formal logarithm of the Ihara zeta function as a counting function for such closed walks. Alternatively, one may take this expression as a starting point of the proof, reduce the problem in (i) to that of counting closed walks of length $<|E|$, and use Kelly's Lemma \ref{Kellys}. The proof of (iii) uses the Jordan normal form decomposition for the matrix $T$, the non-vanishing of an associated ``confluent alternant'' determinant and the fact that $T$ has a ``large'' semi-simple part to reduce the counting problem to length $<|E|$, which then again is done by purely combinatorial means. In case of non-closed walks, this also involves identities based on decomposition of walks into closed and non-returning walks. On the way, we prove some further spectral properties of $T$, e.g., that it is symmetric w.r.t.\ an indefinite quadratic form (Proposition \ref{krein}), and we give an explicit description of its $\pm 1$ eigenspaces in terms of certain spaces of cycles on the graph (Propositions \ref{h} and \ref{hplus}). We also point out that the presence of end-vertices in the graph leads to a non-semi-simple $T$-operator (Proposition \ref{sinkjordan}), so $T$ is, in general, not diagonalisable.  Part (iv) follows from studying Ces\`aro averages of powers of non-negative matrices. Finally, part (v) follows by using an identity of Jacobi for $2\times 2$-sub-determinants.

Two \emph{open problems} that arise from the proofs and that we want to highlight are the following: (a) can the Ihara zeta function $\zeta_G$ be reconstructed from the (multi-)set $\{\zeta_{G-e} \colon e \in E \}$ of Ihara zeta functions of edge-deleted graphs?; (b) for $|E| \geq 2$, is $T$ semi-simple if and \emph{only if} $G$ has an end-vertex?.

We finish this introduction by listing some applications. 
As we explain in \cite{CK} (cf.\ also \cite{CM}), the invariants that we have reconstructed play a central role in the  measure-theoretical study of the action of the fundamental group on the boundary of the universal covering tree of the graph. More precisely, the fundamental group $\Gamma$ of $G$, a free group of rank the first Betti number $b>1$ of $G$, acts on the boundary of the universal covering tree of $G$. This dynamical system ``remembers'' only $b$, since it is topologically conjugate to the action of the free group of rank $b$ on the boundary of its Cayley graph. However, the graph is uniquely determined by a measure on the boundary, namely, the pull-back of the Patterson-Sullivan measure for the action of $\Gamma$ on the boundary. For this measure, the boundary has Hausdorff dimension $\log \lambda$, where $\lambda$ is the Perron-Frobenius eigenvalue of $T$, and the measure itself is expressed on a set of generators for $\Gamma$ in terms of $\lambda$, the entries of the Perron-Frobenius eigenvector of $T$, and the lengths of the loops corresponding to the generators. 

In \cite{SpectralRedemption}, the operator $T$ is used for spectral algorithms that detect clustering in large graphs. This is a hard problem if the graphs under consideration are sparse with widely varying degrees, and the authors argues that use of the operator $T$ outperforms classical algorithms based the spectrum of the adjacency or Laplacian operator. Since the input for their clustering algorithm consists of the two leading eigenvalues of $T$, our main theorem shows reconstruction of this input (if $\bar d \geq 4$). 

In the theory of evolution of species, it has recently been argued that evolutionary relations are not always tree-like \cite{Iersel}. Thus, the phylogenetic reconstruction problem should be considered in the context of general multigraphs, rather than the more traditional case of trees, and our theorem gives a theoretical underpinning for this more general question of reconstruction.

\section{A lemma on polynomial coefficients}\label{poly}

\begin{notation} If $P$ is a single valued polynomial in the variable $\lambda$, let $[\lambda^d]P$ denote the coefficient of $\lambda^d$ in $P$. 
\end{notation} 

\begin{theorem} \label{tfirst} For $d=|E|+1,\dots,2|E|$, the coefficients $[\lambda^d]\det(\lambda-T_G)$ of the characteristic polynomial of the edge adjacency matrix $T_G$ of a graph $G$ are reconstructible from the edge deck $\mathcal{D}^e(G)$. More precisely, 
\begin{equation} \label{booh}
[\lambda^d]\det(\lambda-T_G) = \sum_{r =1}^{\lfloor \frac{d}{2}\rfloor} (-1)^{r+1} \sum_{i_1<i_2<\dots<i_r} [\lambda^{d-2r}]\det(\lambda-T_{G-e_{i_1}\dots-e_{i_r}}). 
\end{equation}
\end{theorem}

\begin{proof}
Let $m=|E|$, and order the rows and columns of the $2m \times 2m$ matrix $T_G$ so that for all $e \in E$, the two orientations $\ra{e}$ and $\la{e}$ label adjacent columns and rows. Set $$\underline{\lambda}=\mathrm{diag}(\lambda_1,\lambda_1,\lambda_2,\lambda_2,\dots,\lambda_m,\lambda_m)$$ and consider the multi-variable polynomial $$P_G(\lambda_1,\dots,\lambda_m):=\det(\underline{\lambda}-T_G).$$ By construction, $P_G$ has at most degree $2$ in each of the individual variables, and after specialisation of all variables to the same $\lambda$, we find $\det(\lambda-T_G)$. The theorem follows by applying the following lemma to $P=P_G$, observing that 
the formula for the expansion of a determinant by $(2m-2)\times(2m-2)$-minors implies
\begin{equation*} \frac{\partial^2P_G}{\partial\lambda_i^2}(\lambda_1,\dots,\lambda_m) = 2 P_{G-{e_i}}(\lambda_1,\dots,\widehat{\lambda_i},\dots,\lambda_m), \end{equation*}
 which we use iteratively to make the replacement \begin{equation} \label{rep} [\lambda^{d-2r}] \frac{\partial^{2r}{P_G}}{\partial\lambda_{i_1}^2\cdots \partial \lambda_{i_r}^2}(\lambda,\dots,\lambda)=2^r  [\lambda^{d-2r}]\det(\lambda-T_{G-e_{i_1}\dots-e_{i_r}}) \end{equation} in (\ref{co}). 
\end{proof}

\begin{lemma} \label{pol}
Let $P(\lambda_1,\dots,\lambda_m)$ denote a polynomial of total degree $2m$ in $m$ variables $\lambda_1,\dots,\lambda_m$. Assume that $P$ is at most quadratic in each individual variable $\lambda_i$.  If $d>m$, then
\begin{equation} \label{co} 
[\lambda^d]P(\lambda,\dots,\lambda) = \sum_{r =1}^{\lfloor \frac{d}{2}\rfloor} (-1)^{r+1} 2^{-r} \sum_{i_1<i_2<\dots<i_r} [\lambda^{d-2r}] \frac{\partial^{2r}{P}}{\partial\lambda_{i_1}^2\cdots \partial \lambda_{i_r}^2}(\lambda,\dots,\lambda). 
\end{equation}
\end{lemma}

\begin{proof}
Since the statement is linear in $P$, it suffices to prove (\ref{co}) if $P$ is a monic monomial and $d=\deg P(\lambda,\dots,\lambda)$ (since for other $d$, the left and right hand side are both zero), when the left hand side is $1$. Suppose that such a monomial $P$ contains exactly $k$ quadratic factors $\lambda^2_i$. Since we assume $d>m$, we have $k \geq 1$, and since $P$ has degree $d$, we also have $k \leq d/2$. Then the right hand side equals $$\sum_{r =1}^{k} (-1)^{r+1} 2^{-r} \cdot 2^r \binom{k}{r} = 1-(1-1)^k = 1.$$
\end{proof}

\section{A formula of Bass and reconstruction of $\zeta_G$}\label{Sbass}

If the graph $G$ under consideration is $(q+1)$-\emph{regular} for some $q \in \Z_{\geq 2}$ (when the reconstruction problem is easy), the Ihara zeta function satisfies functional equations, for example (\cite{Bass}, II.3.10)
$$ \zeta_G(\frac{1}{qu}) =\left(\frac{1-u^2}{1-q^2u^2}\right)^{n\frac{q-1}{2}}q^{qn} u^{(q+1)n} \zeta_G(u). $$
This implies ``palindromic'' relations between the top $m$ and bottom $m$ coefficients of $\zeta^{-1}_G(u)$, so that reconstruction of half the coefficients would be enough for full reconstruction. In the general irregular case that we consider here, there is no such functional equation, but as a substitute for finding relations between the coefficients, at the cost of assuming a certain minimal average degree, we will use an identity of Bass (\cite{Bass}, II.1.5), stating that \begin{equation} \label{bass} \det(1-Tu) = (1-u^2)^{|E|-|V|} \det(1-Au+(D-1)u^2),\end{equation} where $A$ is the adjacency matrix of $G$ and $D = \mathrm{diag} (\deg(v_1),\dots,\deg(v_{|V|}))$ is the degree matrix of $G$. (Recall our convention to denote a unit square matrix of suitable size simply by ``$1$''.) 

\begin{lemma} \label{po} The coefficients $[\lambda^d]B_G$ of $B(\lambda)=\det(\lambda^2-A\lambda+(D-1))$ are edge-reconstructible for $d=2|V|-|E|+1,\dots,2|V|.$
\end{lemma}

\begin{proof} Set $P(\lambda)=\det(\lambda-T)$, and $A(\lambda)=(\lambda^2-1)^{|E|-|V|}$. The identity of Bass becomes $P(\lambda)=A(\lambda) B(\lambda).$ 
All coefficients $[\lambda^i]A$ are easily computable and depend only on $|E|$ and $|V|$; also note that for even $i$, they are non-zero. Now $|V|$ and $|E|$ are edge-reconstructible as $|V|=|V-e|$ and $|E|=|E-e|+1$ for any $e \in E$. The previous theorem implies that the coefficients $[\lambda^k]P$ are edge reconstructible  for $k=|E|+1,\dots,2|E|$. We will use this to reconstruct the coefficients $[\lambda^d]B$ for $d = 2|V|-|E|+1,\dots,2|V|$. We use the formula \[ [\lambda^k]P=\sum\limits_{i=0}^{2|V|} [\lambda^i]B \cdot [\lambda^{k-i}] A.\]
recursively. For $k=2|E|$ we find the relation $$[\lambda^{2|E|}]P=[\lambda^{2|V|}]B \cdot [\lambda^{2(|E|-|V|)}]A,$$ from which we find $[\lambda^{2|V|}]B$. We continue with $[\lambda^{2|E|-1}]P, [\lambda^{2|E|-2}]P,\dots$ and note that in each step corresponding to $[\lambda^{2|E|-j}]P$ we find recursively that the only unknown term in the above sum is $$[\lambda^{2|V|-j}]B \cdot [\lambda^{2(|E|-|V|)}]A.$$ Since $ [\lambda^{2(|E|-|V|)}]A \neq 0$, this allows us to recover $[\lambda^{2|V|-j}]B$. The procedure terminates at $[\lambda^{2|V|-|E|+1}]B$, since $[\lambda^{2|E|-(|E|-1)}]P$ is the highest coefficient which is not reconstructed by the previous theorem. 
\end{proof}

\begin{theorem}[Theorem \ref{main}\textup{(i)}] Let $G$ denote a graph of average degree $\bar d \geq 4$; then the Ihara zeta function $\zeta_G$ of $G$, or, equivalently, the spectrum of the edge-adjacency matrix $T$, is edge-reconstructible.
\end{theorem}

\begin{proof} We first observe that $$[\lambda^0]B = \det(D-1) = \prod_{v \in V} (\deg(v)-1)$$ is reconstructible, since the degree sequence is reconstructible (\cite{Bondy}, Corollary  6.14.(a)), \ref{appcor}(\ref{c3}). Therefore, from the previous lemma, we can reconstruct \emph{all} $[\lambda^d]B$ (and hence $B$, and hence $P$) if $$2|V|-|E|+1\leq 1.$$ This holds exactly if $\bar d\geq 4$, since
$\bar d=2|E|/|V|.$
\end{proof}

\begin{notation} \label{pfnot} If the graph $G$ is connected with no degree one vertices and first Betti number $b_1 \geq 2$ (which follows from our running hypothesis $\bar d \geq 4$), the matrix $T$ is irreducible (\cite{Terras}, 11.10), hence it has a simple Perron-Frobenius eigenvalue $\lambda_{\PF}$ equal to the spectral radius of $T$; it is the maximal real positive eigenvalue of $T$ (e.g.,  \cite{Meyer}, 8.3). 
\end{notation}

\begin{corollary} \label{recpf}
Let $G$ denote a graph of average degree $\bar d \geq 4$; then the Perron-Frobenius eigenvalue $\lambda_{\PF}$ of $T=T_G$ is edge-reconstructible. \qed
\end{corollary}

In Section \ref{walks}, we will give another proof of Theorem \ref{main}(i) that avoids Lemma \ref{pol}, but has the disadvantage of not leading directly to the formula from Theorem \ref{tfirst} for the coefficients in terms of coefficients corresponding to edge-deleted subgraphs. 

In analogy to the question whether the characteristic polynomial of $G$ is determined uniquely by those of its vertex deleted subgraphs \cite{Gutman}, one may ask

\begin{question} Can $\zeta_G$ be reconstructed from $\{ \zeta_{G-e} \colon e \in E \}$?
\end{question} 

\begin{theorem} If $G$ has average degree $\bar d > 4$, then $\zeta_G$ is uniquely determined by the multiset $\mathcal{Z}(G):=\{ \zeta_{G-\mathbf{e}} \colon \emptyset \neq \mathbf e \subset E \},$ where $\mathbf{e}$ runs over all non-empty subsets of $E$. 
\end{theorem}

\begin{proof}
The number $|\mathbf e|$ of distinct edges in $\mathbf{e}$, is determined by the degree of $\zeta^{-1}_{G-\mathbf{e}}$. The formula in Theorem \ref{tfirst} can be rewritten as
$$ [\lambda^d] \zeta_G^{-1} =  \sum_{r =1}^{\lfloor \frac{d}{2}\rfloor} (-1)^{r+1} \sum_{|\mathbf e|=r} [\lambda^{d-2r}] \zeta_{G-\mathbf{e}}^{-1}$$ for $d>|E|$, which is reconstructible from $\mathcal Z (G)$.  As in Lemma \ref{po}, we can then also reconstruct all coefficients of $\zeta_G^{-1}$, as soon as $2|V|-|E|+1<1$, i.e., $\bar d > 4$. \end{proof} 

If $\mathcal Z (G)$ uniquely determines $\det(D-1)$, then one may replace the bound $\bar d > 4$ in this theorem by $\bar d \geq 4$. 

\begin{remark} We list some invariants and properties that have been shown to be determined by $\zeta_G$: 
\begin{enumerate}
\item the \emph{girth} $g$ (length of shortest cycle) of $G$ (since $[\lambda^i]\det(\lambda-T)=0$ for $i=2m-1,\dots,2m-g+1$, and for $i=2m-g,\dots,2m-2g+1$, it  is negative twice the number of $(2m-i)$-gons in $G$, cf.\ Scott and Storm \cite{StormInvolve});  
\item whether $G$ is  \emph{bipartite and cyclic, bipartite non-cyclic or non-bipartite} (Cooper \cite{Cooper}, Theorem 1); 
\item whether or not $G$ is \emph{regular}; and if so, its regularity and the spectrum of its (vertex) adjacency operator (Cooper \cite{Cooper}, Theorem 2).
\end{enumerate}
It follows from our theorem that  for $\bar d \geq 4$, these invariants and properties are edge-reconstructible; but notice that the edge-reconstructibility of these invariants was already known in general from Kelly's Lemma below. \end{remark}

\begin{kellylemma} \label{Kellys}
For any graph $H$ with strictly less vertices than the graph $G$,  the number of induced subgraphs of $G$ isomorphic to $H$ is edge-reconstructible. 
\end{kellylemma}

For Kelly's original 1957 lemma  for vertex reconstruction, see (\cite{Bondy}, Lemma 2.3) (cf.\ \cite{Kelly}, Lemma 1); for the edge version, see  (\cite{Bondy}, Lemma 6.6); and for the multigraph version see \ref{appcor}(\ref{c1}).  We will make repeated use of this result later on.

\section{Symmetry of the edge adjacency operator}

The matrix $T$ is not symmetric in general: $T$ being a symmetric matrix means that $T_{\ra{e_2},\ra{e_1}}=1$ whenever $T_{\ra{e_1},\ra{e_2}}=1$), so for a graph $G$ in our sense, this only happens if $G$ is a ``banana graph'' consisting of two vertices connected by several edges. However, $T$ does have a certain symmetry. 
 
\begin{definition}\label{haakje}
Let $M^\intercal$ denote the transpose of a matrix $M$. Define an indefinite symmetric bilinear form $\langle \cdot,\! \cdot \rangle$ on $\R^{2|E|}$ by $$\langle x,y \rangle := x^{\intercal} J y,$$ where $J$ is a block matrix 
$$ J = \left( \begin{array}{cc} 0 & 1_{|E|} \\ 1_{|E|} & 0 \end{array} \right). $$ The signature of this form is $(|E|,|E|)$, and $(\R^{2|E|}, \langle \cdot ,\! \cdot \rangle)$ is a finite dimensional Kre\u{\i}n space (i.e., an indefinite metric space, compare \cite{Bognar}). 
\end{definition}

An eigenvalue is called \emph{semi-simple} if its algebraic multiplicity (its multiplicity as a root of the characteristic polynomial) and its geometric multiplicity (the dimension of its eigenspace) are equal. 

\begin{proposition} \label{krein} 
The operator $T \colon \R^{2|E|} \rightarrow \R^{2|E|}$ is symmetric for an (indefinite) metric $\langle \cdot ,\! \cdot \rangle$ of signature $(|E|,|E|)$. Its generalized eigenspaces are mutually orthogonal for this metric, and $T$ has at most $|E|$ non-semi-simple eigenvalues. 
\end{proposition}

\begin{proof} 
Observe that $$T_{\ra{e_1},\ra{e_2}} = T_{\la{e_2},\la{e_1}}$$ for all $e_1, e_2 \in E$. By enumerating the rows and columns of $T$ as $\ra{e_1},\dots,\ra{e_{|E|}},\la{e_1},\dots,\la{e_{|E|}}$, we see that $T$ is of the form 
$$ T = \left( \begin{array}{cc} A & B \\ C & A^\intercal \end{array} \right) \mbox{ with } B=B^\intercal \mbox{ and } C^\intercal=C. $$
Being of this form is equivalent to the fact that $T$ satisfies an equation
\begin{equation} \label{Tss} T^\intercal = JTJ. \end{equation}
Equation (\ref{Tss}) means exactly that $T$ is symmetric for the form $\langle \cdot ,\! \cdot \rangle$, namely: $\langle Tx, y \rangle = \langle x, Ty \rangle$ for all $x,y \in \R^{2|E|}$.

 Since $T$ is $J$-symmetric, the different generalized eigenspaces are mutually $J$-orthogonal  (\cite{Bognar}, II.3.3). Finally, since $\langle \cdot ,\! \cdot \rangle$ has signature $(|E|,|E|)$, the space $(\R^{2|E|},\langle \cdot ,\! \cdot \rangle)$ is a Pontrjagin $\Pi_{|E|}$-space in the sense of (\cite{Bognar}, Chapter IX). Since $T$ is $J$-symmetric, it follows that the number of distinct non-semi-simple eigenvalues  is less than or equal to $|E|$ (\cite{Bognar}, IX.4.8).
\end{proof} 

A succinct way of expressing the bilinear form is $$\langle v , w \rangle = \sum_{\ra{e} \in \E} v_{\ra{e}} w_{\la{e}}. $$ Not every $\langle \cdot ,\! \cdot \rangle$-symmetric matrix in a Kre\u{\i}n space is diagonalisable (e.g., the matrix $\left( \begin{smallmatrix} 1_{|E|} & 1_{|E|} \\ 0 &1_{|E|} \end{smallmatrix} \right)$ is $J$-symmetric but not semi-simple). It is easy to construct examples of graphs for which $T$ is not semi-simple, if we temporarily drop our assumption that the graph has no end-vertices:  

\begin{proposition}  \label{sinkjordan}
If $G$ is a connected graph with an end-vertex and $|E|>1$, then $T$ is not semi-simple; actually,  zero is an eigenvalue of $T$ with a non-trivial Jordan block. 
\end{proposition} 

\begin{proof}
If $\ra{e}$ is an oriented edge that ends in an end-vertex (so $T_{\ra{e},\ast}=0$ for all $\ast \in \E$), then $\ra{e} \in \ker T$, and if $\ra{e_1}$ is an oriented edge with $t(\ra{e_1}) = o(\ra{e})$ (which exists by connectedness and since $|E|>1$), then $\ra{e_1} \in \ker T^2 - \ker T$. 
\end{proof}

\begin{question} Give necessary and/or sufficient criteria for a (multi-)graph $G$ to have a semi-simple edge-adjacency operator $T$. More specifically, 
is the presence of end-vertices the only obstruction to semi-simplicity?
\end{question} 

\section{The $\pm 1$-eigenspaces of the edge adjacency operator}

In the next two propositions, we show that $T$ has a ``large'' semi-simple quotient described in terms of the cycle space of $G$. 

\begin{notation} Let $H_1(G,\C)$ denote the space of (complex) linear combinations of cycles on $G$; it is a vector space of dimension $b_1$, the first Betti number of $G$, spanned by induced cycles (\cite{Diestel}, 1.9.1). These cycles we write as formal sums $\sum_{e \in I} e$ over subsets $I \subseteq E$ of the edge set.
\end{notation} 

We have the following (see \cite{Horton}, 5.6 or \cite{CLM}, 1.9): 

\begin{proposition} \label{h} If $b_1>1$, the eigenspace $\ker(1-T)$ for $T$ corresponding to the eigenvalue $1$ is isomorphic to the cycle space via
the map $$ \varphi \colon H_1(G,\C) \rightarrow \ker(1-T) \colon \sum_{e \in I} e \mapsto \sum_{e \in I} (\ra{e}-\la{e}). $$ 
\end{proposition}

Since we will use concepts and notation from the (short) proof, we outline it here: 

\begin{proof}
Since $b_1>1$, the multiplicity of the eigenvalue $1$ in the characteristic polynomial of $T$ is equal to the first Betti number $b_1$ (\cite{Bass}, II.5.10(b)(i); \cite{Hashimoto}, 5.26). It follows that $\ker(1-T)$ has dimension $ \leq b_1$. Therefore, it suffices to prove that the map $\varphi$ is well-defined and injective. 

To show well-definedness of the linear map $\varphi$, fix an induced cycle $c=e_1+\dots +e_r$. Assume that we read the indices of the edges $e_i$ occuring in $c$ as indexed by integers modulo $r$. 
\begin{figure}[h] 
\begin{tikzpicture}
 \draw node at (0,-0.5) {$c$}; 
  \draw node at (0,0.5) {$v$}; 
  \draw node at (-0.8,0.7) {$e_i$}; 
   \draw node at (0.9,0.7) {$e_{i+1}$}; 
    \draw node at (0,2.7) {$B_v$};
    \draw [thick, decorate,decoration={brace, amplitude=8pt}] (-1,2.1) -- (1,2.1);
   \draw[ultra thick] (0,1) -- (-0.7,1.7);
   \draw[ultra thick] (0,1) -- (0.7,1.7);
    \draw[ultra thick] (0,1) -- (0,2);
  \draw[->,ultra thick,dashed] (-1,0) -- (-0.5,0.5); \draw[ultra thick,dashed] (-0.5,0.5) -- (0,1);
     \draw[->, ultra thick, dashed] (0,1) -- (0.5,0.5);  \draw[ultra thick,dashed] (0.5,0.5) -- (1,0);
     \draw[ultra thick,dashed] (-1,-1) -- (-1,0);
     \draw[ultra thick,dashed] (1,-1) -- (1,0);
   \draw (-1,0)  node[circle, inner sep = 0pt, minimum height=2mm, draw, fill = lightgray] {};
      \draw (0,2)  node[circle, inner sep = 0pt, minimum height=2mm, draw, fill] {};
      \draw (-0.7,1.7)  node[circle, inner sep = 0pt, minimum height=2mm, draw, fill] {};
      \draw (0.7,1.7)  node[circle, inner sep = 0pt, minimum height=2mm, draw, fill] {};
         \draw (1,0)  node[circle, inner sep = 0pt, minimum height=2mm, draw, fill = lightgray] {};
          \draw (0,1)  node[circle, inner sep = 0pt, minimum height=2mm, draw, fill] {};
\end{tikzpicture}
    \caption{The ``bush'' of edges $B_v$ at the vertex $v$, w.r.t.\ a cycle $c=\dots + e_i+e_{i+1}+\dots$}
    \label{fig0}
\end{figure}
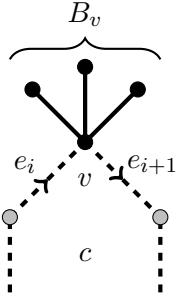

For a vertex $v \in e_j$, let $$B_v = \sum_{\substack{o(\ra{e})=v \\ e \notin c}}  \ra{e}$$ denote the ``bush'' of edges outside the cycle $c$ emanating from the origin of $\ra{e}$ (see Figure \ref{fig0}). Note that if $e \in c$, then $B_{t(e_i)} =B_{o(e_{i+1})}$.

Now 
\begin{align*} T  \left(\sum (\ra{e_i}-\la{e_i})\right)  & = \sum \left(\ra{e}_{i+1} + B_{t(e_i)} - \la{e}_{i-1} - B_{o(e_i)}\right) = \sum (\ra{e_i}-\la{e_i}),  \end{align*}
so indeed, $\varphi(c) \in \ker(1-T)$. 
Finally, the injectivity of $\varphi$ follows immediately from the linear independence of the elements $\ra{e},\la{e}$ (for $e \in E$) in the space $\C^{2|E|}$ on which the operator $T$ acts: if $$\sum_{e \in E} a_{\ra{e}} \ra{e} - \sum_{e \in E} a'_{\la{e}} \la{e} = 0,$$ for some $a_{\ast} \in \C$, then $\sum a_{\ra{e}} e = 0$, so only the zero  cycle is mapped to zero. 
\end{proof} 

\begin{remark} If $b_1=1$, the map $\varphi$ is not an isomorphism, but can still be described in terms of edges (\cite{CLM}, 1.14). Since we assume $\bar d \geq 4$, we have $b_1 =|E|-|V|+1 \geq |V|+1>1$.  \end{remark}

Next we consider the eigenspace of eigenvalue $-1$.
 
\begin{notation} The integer $p$ is defined by $p=0$ if $G$ is bipartite and $p=1$ otherwise. Let $H_1^+(G,\C)$ denote the subspace of $H_1(G,\C)$ generated by cycles of \emph{even} length. 
\end{notation} 
We have $$H_1(G,\C) = H^+_1(G,\C) \oplus \C^p.$$ Indeed, a graph is bipartite if and only if all cycles are even (\cite{Diestel}, 1.6.1), and if the graph is not bipartite, let $c_1,\dots,c_r,c_{r+1},\dots c_{b_1}$ denote a basis for its cycle space based at a common vertex $v_0$, in which the first $r$ cycles are even and the remaining are odd. Then $$c_1,\dots,c_r,c_r+c_{b_1},\dots,c_{b_1-1}+c_{b_1},c_{b_1}$$ is a basis in which the first $b_1-1$ cycles are even and the final one is not. 

\begin{proposition} \label{hplus} For every even cycle $c=\sum_{e \in I} e$, choose a proper 2-coloring $\kappa_c \colon I \rightarrow \{\pm 1\}$ of the edges of $c$. Then the map $$ \psi \colon H_1^+(G,\C) \rightarrow \ker(1+T) \colon c=\sum_{e \in I} e \mapsto \sum_{e \in I} \kappa_c(e) (\ra{e}+\la{e}) $$ is an isomorphism of complex vector spaces.
\end{proposition}

\begin{proof}

The multiplicity of the eigenvalue $-1$ in the characteristic polynomial of $T$ is $b_1-p$ (\cite{Bass}, II.5.10(b)(ii); \cite{Hashimoto}, 5.32). It suffices to prove that the map $\psi$ is well-defined and injective. For well-definedness, fix an induced even cycle $c=e_1+\dots +e_r$ as before. Without loss of generality, we can assume $\kappa_c(e_j) = (-1)^j$. Then, using the notation for ``bushes'' from the proof of Proposition \ref{h}, we find
\begin{align*} T & \left(\sum_{2 \mid i}(\ra{e_i}+\la{e_i}) - \sum_{2 \nmid i}(\ra{e_i}+\la{e_i})  \right) \\ & = \sum_{2 \mid i} (\ra{e}_{i+1} + B_{t(e_i)} + \la{e}_{i-1} + B_{o(e_i)}) -  \sum_{2 \nmid i} (\ra{e}_{i+1} + B_{t(e_i)} + \la{e}_{i-1} + B_{o(e_i)})   \\ & =  \sum_{2 \mid i} (\ra{e}_{i+1} + \la{e}_{i-1} ) -  \sum_{2 \nmid i} (\ra{e}_{i+1} + \la{e}_{i-1})  \\ & =  \sum_{2 \nmid j} (\ra{e}_{j} + \la{e}_{j-2} ) -  \sum_{2 \mid j} (\ra{e}_{j} + \la{e}_{j-2}) \\ & = - \left(\sum_{2 \mid i}(\ra{e_i}+\la{e_i}) - \sum_{2 \nmid i}(\ra{e_i}+\la{e_i})  \right),  \end{align*}
so $\psi$ is well-defined. The injectivity of $\psi$ follows again from the linear independence of the elements $\ra{e},\la{e}$ (for $e \in E$). 
\end{proof}

\begin{corollary} \label{mul} The eigenvalues $\pm 1$ are semi-simple for the operator $T$, of respective multiplicities $|E|-|V|+1$ and $|E|-|V|+1-p$. \qed
\end{corollary}

There are examples (such as the complete 4-graph with one edge deleted \cite{Terras}, Example 2.8) in which all other eigenvalues of $T$, apart from $\pm 1$, are simple and semi-simple. This shows that one cannot expect a more general statement than \ref{mul} concerning multiplicities of eigenvalues of $T$. 

\section{Reconstruction of closed non-backtracking walks} \label{walks}

For a positive integer $r$, the entry of $T^r$ at place $\ra{e_1}, \ra{e_2}$ is the number of non-backtracking walks that start in the direction of the oriented edge $\ra{e_1}$ and end at the oriented edge $\ra{e_2}$. Let $$N_r(\ra{e})=T^r_{\ra{e},\ra{e}}$$ denote the number of closed such walks through an oriented edge $\ra{e} \in \mathbf{E}$. Observe that by symmetry (``walking backwards''), $N_r(\ra{e})=N_r(\la{e})$. 
For an unoriented edge $e \in E$, $N_r(e)=2N_r(\ra{e})$ (for any choice $\ra{e}$ of orientation on $e$),  denotes the number of oriented non-backtracking closed walks that pass through $e$. The total number of non-backtracking unoriented closed walks of length $r$ in $G$ is $$N_r = \sum_{e \in E} N_r(e) = \mathrm{tr}(T^r)=\sum_{\ra{e} \in \mathbf{E}} T^r_{\ra{e},\ra{e}}, $$ where $\mathrm{tr}$ denotes the trace of a matrix. 

\begin{theorem}[Theorem \ref{main}\textup{(ii)}] Let $G$ denote a graph of average degree $\bar d \geq 4$; then the number of non-backtracking closed walks on $G$ of given length is edge-reconstructible. 
\end{theorem} 

\begin{proof}The claim follows directly from the formal power series identity 
\begin{equation} \label{logdet} \log \zeta_G = - \log \det(1-uT) = \sum_{n \geq 1} \frac{\mathrm{tr}(T^n)}{n} u^n \end{equation} (easily proven by triagonalizing the matrix $T$ over $\C$) and part (i) of the theorem. 
\end{proof}

We now refine this result, in analogy with the vertex situation studied by Godsil and McKay in \cite{Godsil} (but our proofs are rather different, since we do not have a semi-simple operator and we cannot rely on reconstruction results for complementary graphs). 

\begin{remark} \label{mn} We define the value of $N_r$ (and other similar functions) at an element $H=G-e \in \mathcal D^e(G)$ of the edge deck to be equal to $N_r(e)$. Since $\mathcal D^e(G)$ is a multiset,  it is possible that $G-e \cong G-e'$ for two different edges $e$ and $e'$. Our methods of proof imply that the value $N_r(e)$ only depends on the isomorphism type of $H$, not on the edge $e$, and thus, $N_r$ is well-defined on the edge deck.  
\end{remark} 

\begin{theorem}[Theorem \ref{main}\textup{(iii)(a)}] Let $G$ denote a graph of average degree $\bar d>4$. Then the function $N_r \colon \mathcal{D}^e(G) \rightarrow \Z$ that associates to an element $G-e$ of the edge deck $\mathcal{D}^e(G)$ of $G$ the number of  non-backtracking closed walks on $G$ of given length passing through $e$ is edge-reconstructible.
\end{theorem}

\begin{proof}
As a first step, we use the Jordan normal form of $T$ to prove the following: 

\begin{lemma} \label{J} The values $N_r(e)$ for all $r \in \Z_{\geq 0}$ are uniquely determined by the values $N_r(e)$ for $r \leq M-1$, where $M$ is the sum of the maximal sizes of Jordan blocks for the different eigenvalues of $T$.
\end{lemma} 

\begin{proof}[Proof of Lemma \ref{J}] Suppose that $T$ has $N$ distinct eigenvalues $\lambda_1,\dots,\lambda_N$. Let $m_i$ denote the multiplicity of $\lambda_i$. Suppose that $\lambda_i$ occurs in $\ell_i$ different Jordan blocks, and let $\mu_{i,j}$ denote the size of the $j$-th such block ($j=1,\dots,\ell_i$), so that $m_i = \sum_j \mu_{i,j}$.  Let $P$ denote the matrix whose columns are a complete set of generalized eigenvectors for $T$, then $T=P\Lambda P^{-1}$, where $\Lambda$ is a Jordan normal form of $T$.   Fix an (oriented) edge $\ra{e}$. All vectors will depend on $\ra{e}$, but, for readability, we will mostly suppress it from the notation. 
If $x_{\ra{e}}$ is the $2|E|$-column vector with a $1$ in place $\ra{e}$ and $0$ elsewhere, then 
\begin{equation} \label{nre} N_r(\ra{e}) = x_{\ra{e}}^{\intercal} T^r x_{\ra{e}} =  v \Lambda^r v', \end{equation}
where $v=x_{\ra{e}}^{\intercal} P$ and $v'= P^{-1} x_{\ra{e}}$. 

Expanding the powers of the Jordan normal form, we find that 
\begin{equation} \label{fff} N_r(\ra{e}) = \sum_{i=1}^N \sum_{j=1}^{\ell_i} \sum_{k=0}^{\mu_{i,j}-1} \lambda_i^{r-k} \binom{r}{k} w_{i,j,k}  \end{equation}
for some constants 
$$ w_{i,j,k} = \sum_{l=a_{i,j}}^{a_{i,j}+\mu_{i,j}-k} v_l v'_{l+k}, \mbox{ with } a_{i,j}:=\sum_{\substack{i_0 \leq i \\ j_0<j}} \mu_{i_0,j_0}. $$
Let $$M_i:= \max \{ \mu_{i,j} \colon j \} \mbox{ and } M=\sum_{i=1}^N M_i.$$ Set new variables $w_{i,j,k}=0$ when $k \geq \mu_{i,j}$; with this convention, we can replace the third summation in (\ref{fff}) by $k=0,\dots,M_i-1$, independent of $j$. We then collect terms in $j$, to find that there exists constants $y_{i,k}$ such that 
\begin{equation} \label{ttt} N_r(\ra{e}) = \sum_{i=1}^N \sum_{k=0}^{M_i-1} \lambda_i^{r-k} \binom{r}{k} y_{i,k}; \end{equation}
namely, $$y_{i,k}:= \sum_{j=1}^{l_i} w_{i,j,k}. $$
The set of equations (\ref{ttt}) can be written in matrix form as 
$$\mathbb{V} Y = \mathbf{N}, $$
 where $Y$ is a column vector consisting of $y_{i,k}$, $\mathbf{N}$ is a column vector with entries $N_i(\ra{e})$ for $i=0,\dots,M-1$, and $\mathbb V$ is the $M \times M$-matrix given as concatenation $$\mathbb V = (\mathbb V_1 | \mathbb V_2 | \dots | \mathbb V_N)$$ with 
 $\mathbb V_i$ an $M \times M_i$ matrix with entries 
 $$ (\mathbb V_i)_{k,l} = \binom{k-1}{k-l-1} \lambda_i^{k-1-l} . $$ 
 Note that $\mathbb V$ is edge-reconstructible by our reconstruction of the spectrum of $T$. 
If $T$ is semi-simple, this is a classical Vandermonde matrix. In general, it is a Vandermonde matrix with inserted columns corresponding to powers of the nilpotent part of $T$; it is the matrix consisting of generalized eigenvectors for the companion matrix of the characteristic polynomial of $T$ and historically known as a ``confluent alternant'' \cite{Kalman}. We have (loc.\ cit., Formula (14))
$$ \det \mathbb V = \pm \prod_{i<j} (\lambda_i-\lambda_j)^{M_i \cdot M_j} \neq 0, $$
 and hence $\mathbb V$ is invertible. Therefore, $Y$ is uniquely determined by $\mathbf{N}$, and $N_r(e)$ is uniquely determined for all $r$ by its values for $r \leq M-1$. 
 \end{proof} 
 
As a second step, we prove that for $\bar d >4$, $M-1 <|E|$. Indeed, recall from Corollary \ref{mul} that $T$ has semi-simple eigenvalue $\lambda_1 = +1$ with multiplicity $|E|-|V|+1$ and semi-simple eigenvalue $\lambda_2 = -1$ with multiplicity at least $|E|-|V|$. Hence $M_1=M_2=1$ and the number $M$ satisfies \begin{equation} \label{bip} M-1 \leq 2+2|E|-(|E|-|V|)-(|E|-|V|+1)-1 = 2|V|.\end{equation} Since we assume $\bar d = 2|E|/|V|> 4$, we have $M-1 < |E|$. 
  
Finally, we show how to reconstruct $N_r(e)$ for $r < |E|$. Suppose that $\mathcal{G}_{i}$ is the set of isomorphism classes of graphs with $i$ edges. Given a graph $H$, let $P_r(H)$ denote the number of distinct closed non-backtracking walks of length $r$ on $H$ that go through every edge of $H$ (possibly multiple times, with no preferred starting edge).  Let $S(H,G)$ denote the number of induced subgraphs of $G$ isomorphic to $H$. For $r<|E|$, we have
$$ N_r(\ra{e}) = \frac{1}{2} \sum_{\substack{H \in \mathcal{G}_{i} \\ i \leq r}} P_r(H)(S(H,G)-S(H,G-e)). $$
Indeed, $S(H,G)-S(H,G-e)$ is the number of induced subgraphs of $G$ isomorphic to $H$ that pass through $e$. Any closed  non-backtracking walk of length $r$ on $H$, embedded in $G$ to pass through $e$, gives rise to such a walk that starts and ends at $e$ in a given direction (for both chosen directions). 

By Kelly's Lemma \ref{Kellys}, since $H$ has less than $|E|$ edges, the right hand side is reconstructible, hence so is the left hand side.

This finishes the proof of the theorem that $N_r(\ra{e})$ is edge-reconstructible for all $r$. 
\end{proof}

\begin{proposition} \label{BP} If $G$ is bipartite of average degree $\bar d \geq 4$, the function $N_r$ is edge-reconstructible for all $r>0$. 
\end{proposition}

\begin{proof}
If $G$ is bipartite, then the eigenvalue $-1$ also has multiplicity $|E|-|V|+1$ (cf.\ Corollary \ref{mul}), so the estimate $M-1<|E|$ in Equation (\ref{bip}) holds even if $\bar d = 4$. 
\end{proof}

 We now give another proof of part (i) of Theorem \ref{main} along the lines of the previous proof, which has a more combinatorial flavour  and avoids using Lemma \ref{pol} (but does not lead directly to the inductive formula from Theorem \ref{tfirst}). 

\begin{proof}[Second proof of Theorem \ref{main}(i)] 
The result of Bass (\cite{Bass}, II.5.4) says that we can write $$\det(1-Tu) = (u-1)^{|E|-|V|+1}(u+1)^{|E|-|V|}D^+(u)$$ for some polynomial $D^+(u)$ of degree $2|V|-1$ with $D^+(0) \neq 0$. Plugging this into the generating series (\ref{logdet}) and take logs, we find
$$ (|E|-|V|+1) \sum_{j \geq 1} \frac{u^{j}}{j} +(|E|-|V|) \sum_{j \geq 1} \frac{(-u)^{j}}{j} - \log D^+(u) = \sum_{r \geq 1} N_r \frac{u^r}{r}. $$
It follows that we know the entire polynomial $\det(1-Tu)$ as soon as we know $D^+(u)$, which happens as soon as we know $N_r$ for all $r \leq 2|V|-1$. With $\bar d = 2|E|/|V| \geq 4$, we need to reconstruct $N_r$ for $r <|E|$. But this can be done using Kelly's Lemma \ref{Kellys}, as follows:
$$N_r = \sum_{\substack{H \in \mathcal{G}_{i} \\ i \leq r}} P_r(H) S(H,G), $$
where $\mathcal{G}_i, P_r$ and $S(H,G)$ are as in the above proof of Theorem \ref{main}\textup{(iii)}. 
\end{proof} 

For $e \in E$, let $F_r(e)$ denote the number of closed non-backtracking walks that pass through $e$ in both directions at least once. Then $F_r(e) = 2 F_r(\ra{e})$, where for an oriented edge $\ra{e} \in \mathbf{E}$, $F_r(\ra{e})$ is the number of closed  non-backtracking walks that start at $\ra{e}$ and pass through $\la{e}$ at least once. 

\begin{theorem}[Theorem \ref{main}\textup{(v)}]  Let $G$ denote a graph of average degree $\bar d>4$. Then the function $F_r \colon \mathcal{D}^e(G) \rightarrow \Z$ that associates to an element $G-e$ of the edge deck $\mathcal{D}^e(G)$ of $G$ the number of  non-backtracking closed walks on $G$ of given length that pass through $e$ in both directions at least once is edge-reconstructible. 
\end{theorem}

\begin{proof}
First, observe that \begin{equation} \label{f} F_r(\ra{e}) = \sum_{i=0}^r (T^i)_{\ra{e},\la{e}} (T^{r-i})_{\la{e},\ra{e}}. \end{equation}
The edge adjacency matrix $T_{G-e}$  of $G-e$ is the matrix $T$ in which the rows and column corresponding to the edges $\ra{e}$ and $\la{e}$ have been removed. Let $T[e_1,e_2]$ denote the $2\times 2$ matrix in which only the elements in column/row $e_1$ and $e_2$ are preserved. In this situation, Jacobi's identity applied to the matrix $1-uT$ (generalizing from $1\times 1$ minors to $2\times 2$ minors the more familiar formula for an inverse matrix in terms of determinant and adjugate; see e.g., formula (12) in \cite{BS}) states that 
\begin{equation*} \label{quot} \frac{\det(1-uT_{G-e})}{\det(1-uT)} = \det((1-uT)^{-1}[\ra{e},\la{e}]). \end{equation*}
The left hand side of this equation equals $\zeta_G(u)/\zeta_{G-e}(u)$, which is reconstructible by part (i). Since $$(1-uT)^{-1} = \sum_{r \geq 0} u^r T^r,$$ we find that the right hand side equals
\begin{align*} \det ((1-uT)^{-1}[\ra{e},\la{e}])&= \det \left( \begin{matrix}  \sum u^r N_r(\ra{e}) & \sum u^r (T^r)_{\ra{e},\la{e}} \\ \sum u^r (T^r)_{\la{e},\ra{e}} & \sum u^r N_r(\ra{e}) \end{matrix} \right) \\
& = \sum_{r \geq 0} u^r  \left( \sum_{i=0}^r N_i(\ra{e})N_{r-i}(\ra{e})  - F_r(\ra{e}) \right), \end{align*} using the expression for $F_r(\ra{e})$ from (\ref{f}). Since $N_r(\ra{e})$ is edge-reconstructible, we conclude that
the function $F_r(\ra{e})$, and hence $F_r(e)$, is edge-reconstructible. 
\end{proof} 

Similar to Proposition \ref{BP}, we get
\begin{proposition} If $G$ is bipartite of average degree $\bar d \geq 4$, the function $F_r$ is edge-reconstructible for all $r>0$. \qed
\end{proposition}

\section{Reconstruction of non-closed non-backtracking walks } 

We now consider the case of non-backtracking walks between two (possibly different) edges: 

\begin{theorem}[Theorem \ref{main}\textup{(iii)(b)}] Let $G$ denote a graph of average degree $\bar d>4$. Then the function $M_r \colon \mathcal{D}^e(G) \rightarrow \Z$ that associates to an element $G-e$ of the edge deck $\mathcal{D}^e(G)$ of $G$ the number of non-backtracking (not necessarily closed) 
walks on $G$ of given length starting at $e$ (in any direction) is edge-reconstructible.
\end{theorem}

\begin{proof}
Let $M_r(\ra{e})$ denote the number of  non-backtracking walks of length $r$ that start in the direction of $\ra{e}$ (but do not necessarily return to $\ra{e}$). Then, similarly to the expression derived for $N_r(\ra{e})$ in the previous proof, we find

$$ M_r(\ra{e}) = x_{\ra{e}}^{\intercal} T^r \mathbf{1} = \sum_{i=1}^N \sum_{k=0}^{M_i-1} \lambda_i^{r-k} \binom{r}{k} y'_{i,k}, $$
where $\mathbf{1}$ is the $2|E|$-column vector consisting of all $1$'s and $y'_{i,k}$ is an expression similar to $y_{i,k}$ in the previous proof, but with the role of $v'$ taken by $\mathbf{1}$. 
Now  $$M_r(e) = M_r(\ra{e})+M_r(\la{e}) =  \sum_{i=1}^N \sum_{k=0}^{M_i-1} \lambda_i^{r-k} \binom{r}{k} (y'_{\ra{e},i,k}+y'_{\la{e},i,k}),$$ (where we have indicated the dependence of  $y'_{i,k}$ on the oriented edge $\ra{e}$ in the subscript) is the number of non-backtracking walks of length $r$ that start at $e$ in any direction.  The same reasoning as in the  proof of Lemma \ref{J} shows that is suffices to reconstruct $M_r(e)$ for $r < |E|$; namely, we find a matrix equation 
$$\mathbb V Y' = \mathbf{M}, $$
 where $Y'$ is a column vector consisting of $y'_{\ra{e},i,k}+y'_{\la{e},i,k}$, $\mathbf{M}$ is a column vector with entries $M_i(e)$ for $i=0,\dots,M-1$, and $\mathbb V$ is the same (invertible) matrix as in the previous proof. This shows that $Y'$, and hence $M_r(e)$ for all $r$, is determined by $M_r(e)$ for $r \leq M-1 < |E|$. 

Let $W_r(e)$ denote the total number of walks through the edge $e$. This number is reconstructible by Kelly's Lemma \ref{Kellys} for $r<|E|$, since 
$$ W_r(e)  = \sum_{\substack{H \in \mathcal{G}_{i} \\ i \leq r}} Q_r(H)(S(H,G)-S(H,G-e)), $$
where $Q_r(H)$ is the number of (not necessarily closed) walks of length $r$ that pass through every edge of $H$. 

Let $O_r(\ra{e})$ denote the number of walks of length $r$ starting at $\ra{e}$ that never return to $\ra{e}$ (but might go though $\la{e}$), and let $O_r(e)=O_r(\ra{e})+O_r(\la{e})$ denote the number of walks starting in $e$ but never return to $e$ in the same direction. We call these \emph{non-returning walks}. We then have the following relations (similar to the ones for vertex walks discussed in \cite{Godsil}, Formula (1)): 

\begin{enumerate}
\item Every walk of length $r$ through $e$ decomposes as a non-returning walk of length $i$ into $e$, then a closed walk of length $j$ through $e$, followed by a non-returning walk of length $k$ starting at $e$, for $r+2=i+j+k$ (see Figure \ref{ijk}). Hence 
\begin{equation} \label{w1} W_r(e) = \sum_{i+j+k=r+2} O_j(e) N_j(e) O_k(e). \end{equation}
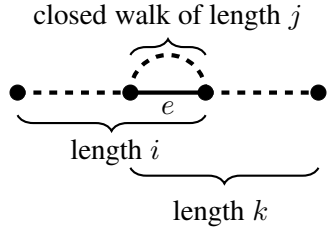
\begin{figure}[h] 
\begin{tikzpicture}
 \draw node at (0,-0.2) {$e$}; 
 \draw node at (-0.7,-0.8) {length $i$}; 
  \draw node at (0.7,-1.6) {length $k$}; 
    \draw node at (0,1) {closed walk of length $j$}; 
   \draw[ultra thick,dashed] (-2,0) -- (-0.5,0);
      \draw[ultra thick,dashed] (0.5,0) -- (2,0);
   \draw[ultra thick] (-0.5,0)--(0.5,0);
   \draw (-2,0)  node[circle, inner sep = 0pt, minimum height=2mm, draw, fill] {};
      \draw (-0.5,0)  node[circle, inner sep = 0pt, minimum height=2mm, draw, fill] {};
         \draw (0.5,0)  node[circle, inner sep = 0pt, minimum height=2mm, draw, fill] {};
           \draw (2,0)  node[circle, inner sep = 0pt, minimum height=2mm, draw, fill] {};
           
          \draw[ultra thick,dashed] (0.5,0) arc (0:180:0.5cm);

\draw [thick, decorate,decoration={brace, mirror, amplitude=6pt}] (-2,-0.3) -- (0.5,-0.3);
\draw [thick, decorate,decoration={brace, mirror, amplitude=6pt}] (-0.5,-1) -- (2,-1);
\draw [thick, decorate,decoration={brace, amplitude = 6pt}] (-0.5,0.5) -- (0.5,0.5);

\end{tikzpicture}
    \caption{Decomposition of a walk through $e$ of total length $i+j+k-2$}
    \label{ijk}
\end{figure} 
\item Every walk of length $r$ starting at $e$ decomposes as a closed walk of length $i$ followed by a non-returning walk of length $j$, where $i+j=r+1$. 
Hence 
\begin{equation} \label{w2} M_r(e) = \sum_{i+j=r+1} N_i(e) O_j(e). \end{equation}
\end{enumerate}
 If we express these relations (\ref{w1}) and (\ref{w2}) using generating series $W(x) = \sum W_r(e) x^r$, etc., they become 
$$ \left\{ \begin{array}{l} W(x) = x^2 N(x) O(x)^2 \\ M(x) = x N(x) O(x), \end{array} \right. $$
 from which we can eliminate $O(x)$, to find 
 $ M(x) = \sqrt{ W(x) N(x) },$ i.e., for all $r \geq 0$: 
 $$ \sum_{i+j=r} M_i(e) M_j(e) = \sum_{i+j=r} W_i(e) N_j(e). $$
 Since we have already reconstructed $N_j(e)$ for all $j$ and $W_i(e)$ for all $i<|E|$, we can use this formula to reconstruct recursively the values $M_r(e)$ for all $r<|E|$. This suffices to reconstruct $M_r(e)$ for all integers $r$. \end{proof} 
 
 Similar to Proposition \ref{BP}, we get
\begin{proposition} If $G$ is bipartite of average degree $\bar d \geq 4$, the function $M_r$ is edge-reconstructible for all $r>0$. \qed
\end{proposition}

 \section{Reconstruction of the Perron-Frobenius eigenvector of $T$}
\begin{notation} Let $\mathbf{p}$ denote the normalized Perron-Frobenius eigenvector corresponding to the (simple) Perron-Frobenius eigenvalue $\lambda_{\PF}$ of $T$, where the normalization is given by $$\langle \mathbf{p}, \mathbf{p} \rangle = \mathbf{p}^{\intercal} J \mathbf{p} = 1$$ in terms of the indefinite metric $\langle \cdot, \! \cdot \rangle$ from Definition \ref{haakje}. Spelled out in coordinates, this means that $\mathbf{p}$ is normalized by 
\begin{equation} \label{norm} 2 \sum_{e \in E} \mathbf{p}_{\ra{e}} \mathbf{p}_{\la{e}} = 1. \end{equation}
\end{notation}

\begin{theorem}[Theorem \ref{main}\textup{(iv)}] \label{recp} Let $G$ denote a graph of average degree $\bar d>4$. Then for any symmetric polynomial $f$ of two variables, 
the function $\mathcal{D}^e(G) \rightarrow \R \colon G-e \mapsto  f(\mathbf{p}_{\ra{e}},\mathbf{p}_{\la{e}})$ is edge-reconstructible.
 In particular, the unordered pairs $\{\mathbf{p}_{\ra{e}},\mathbf{p}_{\la{e}}\}$ are edge-reconstructible. 
\end{theorem}

\begin{proof} It suffices to prove this for $f$ equal to one of the elementary symmetric functions $$\sigma_e := \mathbf{p}_{\ra{e}} + \mathbf{p}_{\la{e}} \mbox{ and } \pi_e:=\mathbf{p}_{\ra{e}}\cdot \mathbf{p}_{\la{e}}.$$ 

The result follows from Perron-Frobenius theory for non-negative matrices (see, e.g., section 8.3 in \cite{Meyer}), as follows. 
Since $T$ is a non-negative irreducible matrix (cf.\ \ref{pfnot}), the so-called Ces\`aro averages of $T$, defined as the left hand side in Equation (\ref{Cesa}), are given by 
\begin{equation} \label{Cesa} \lim_{k \rightarrow +\infty} \frac{1}{k} \sum_{r=0}^{k-1} \frac{T^r}{\lambda_{\PF}^r} = \frac{\mathbf{p}\mathbf{q}^{\intercal}}{\mathbf{q}^{\intercal} \mathbf{p}}, \end{equation} 
where $\mathbf{p}$ and $\mathbf{q}$ are Perron-Frobenius eigenvectors of $T$ and $T^{\intercal}$, respectively (\cite{Meyer} 8.3.2). Notice that $\mathbf{p}$ and $\mathbf{q}$ are determined up to scaling, but different choices do not change the right hand side of the equation. 
Now Formula (\ref{Tss}) implies the following equivalence between left and right eigenvectors $v$ for $T$:  $$Tv=\lambda v \iff v^{\intercal} J T = \lambda v^{\intercal} J.$$ 
Since from $T^{\intercal} \mathbf{q} = \lambda_{\PF} \mathbf{q}$, it follows that $\mathbf{q}^{\intercal} T = \lambda_{\PF} \mathbf{q}^{\intercal}$, we can set $\mathbf p$ to be normalized and $\mathbf q = J \mathbf p$. Hence the expression in (\ref{Cesa}) becomes 
\begin{equation} \label{Cesa2} \lim_{k \rightarrow +\infty} \frac{1}{k} \sum_{r=0}^{k-1} \frac{T^r}{\lambda_{\PF}^r} = \mathbf{p}\mathbf{p}^{\intercal} J, \end{equation} 
since $\mathbf{p}$ is normalized as in (\ref{norm}). 

It follows that 
\begin{equation} \label{PFfs} \lim_{k \rightarrow +\infty} \frac{1}{2k} \sum_{r=0}^{k-1} \frac{N_r(e)}{\lambda_{\PF}^r} = \lim_{k \rightarrow +\infty} \frac{1}{k} \sum_{r=0}^{k-1} \frac{x_{\ra{e}}^{\intercal} T^r x_{\ra{e}}}{\lambda_{\PF}^r}  = \mathbf{p}_{\ra{e}} \mathbf{p}_{\la{e}}=\pi_e. 
\end{equation}
Similarly, we have 
\begin{align} \label{mff} \lim_{k \rightarrow +\infty} \frac{1}{k} \sum_{r=0}^{k-1} \frac{M_r(e)}{\lambda_{\PF}^r} 
& = \lim_{k \rightarrow +\infty} \frac{1}{k} \sum_{r=0}^{k-1} \frac{ x_{\ra{e}}^{\intercal} T^r \mathbf{1}+x_{\la{e}}^{\intercal} T^r \mathbf{1} }{\lambda_{\PF}^r}  \\ & = (\mathbf{p}_{\ra{e}}+\mathbf{p}_{\la{e}}) \sum_{e' \in E} (\mathbf{p}_{\ra{e'}}+\mathbf{p}_{\la{e'}}) =: \tilde\sigma_e, \nonumber \end{align}
with $$ \tilde \sigma_e = \alpha \sigma_e \mbox{ for } \alpha =  \sum_{e' \in E} \sigma_{e'}. $$ 
Hence the numbers $\tilde \sigma_{{e}}$ and $\pi_e$ can be reconstructed from $\mathcal{D}^e(G)$, since the left hand side of the above formulas (\ref{PFfs}) and (\ref{mff}) can. Since the entries of the Perron-Frobenius eigenvector are all non-negative, we find that $\alpha$
is positive. Adding up all terms in (\ref{mff}), we find that 
$$  \sum_{e \in E} \tilde \sigma_e= \alpha^2,$$
hence $\alpha \geq 0$ is determined, and so also $\sigma_e=\tilde \sigma_e/\alpha$ is edge-reconstructible. The final statement follows since the elements of the unordered pair $\{\mathbf{p}_{\ra{e}},\mathbf{p}_{\la{e}}\}$ are the roots of $x^2-\sigma_e x + \pi_e = 0$. 
\end{proof} 

Similar to Proposition \ref{BP}, we get
\begin{proposition} If $G$ is bipartite of average degree $\bar d \geq 4$, the unordered pairs $\{\mathbf{p}_{\ra{e}}, \mathbf{p}_{\la{e}} \}$  are edge-reconstructible for all $r>0$. \qed
\end{proposition}

\appendix

\section{Some results on multigraph edge reconstruction \\  \normalfont by \textsc{Daniel C.\ McDonald}}{\let\thefootnote\relax\footnote{{The subject of edge-reconstructibility of nonsimple multigraphs was broached by Matthew Yancey in the Structure of Graphs course taught by Alexandr Kostochka in Spring 2010 at Illinois, leading to a formulation of Conjecture \ref{ercm}. The appendix contains some of the results obtained during a 2010 REGS in Combinatorics organized by Douglas B.\ West at the University of Illinois, funded by National Science Foundation grant DMS 08-38434 ``EMSW21- MCTP: Research Experience for Graduate Students''.}}}
A multigraph is \emph{simple} if it has no loops or multiedges. Given a multigraph $M$, we will call a simple graph $G$ the \emph{underlying simple graph of $M$} if $G$ is obtained by replacing all multiedges of $M$ with edges of multiplicity 1.

\begin{conjecture}[Reconstruction Conjecture on nonsimple multigraphs] \label{ercm} Every nonsimple multigraph with more than two edges is edge-reconstructible.
\end{conjecture} 
Multigraphs in which every multiedge has the same multiplicity $m$ have the same vertex-deleted subgraphs as their underlying simple graphs, except each edge is replaced by a multiedge of multiplicity $m$, so the Reconstruction Conjecture on nonsimple multigraphs is more difficult than the Reconstruction Conjecture on simple graphs. 

For simplicity's sake $M$ will henceforth refer to a finite loopless nonsimple multigraph with at least three edges and no isolated vertices, and $G$ will always refer to the underlying simple graph of $M$.
For $m \geq 1$ we will refer to a multiedge of multiplicity $m$ as an $m$-edge. We will use the term edge when referring to a $1$-edge or an individual edge making up part of a larger multiedge, and we will refer to multiedges of multiplicity at least $2$ as nontrivial multiedges.
The underlying simple graph $G$ of $M$ is an edge-constructible property ($G$ will be the underlying simple graph of any card with the maximum number of $1$-edges) and therefore classes of multigraphs defined by underlying simple graph structure are edge-recognizable. For a multigraph $Q$ with fewer edges than $M$, counting arguments show that the parameter $S_Q(M)$ is edge-reconstructible, where $S_Q(M)$ counts the times $Q$ appears as a subgraph of $M$.
For a set $X = \{Q_1,\dots,Q_k\}$ of multigraphs let $S_{Q_i}^X(M)$ denote the number of times $Q_i$ appears in
$M$ not as a subgraph of any $Q_j$ for $i \neq j$.
\begin{lemma} Let $X = \{Q_1,\dots,Q_k\}$ be a set of multigraphs each with fewer edges than $M$. Suppose
it can be verified from the deck of $M$ that for any (not necessarily distinct) $a, b, c$, if a copy of $Q_a$ is
contained in the intersection of a copy of $Q_b$ and a copy of $Q_c$, then for some $d$ (potentially one of
$a,b,c$), those copies of $Q_b$ and $Q_c$ are contained in some copy of $Q_d$. Then the parameter $S_{Q_i}^X (M)$ 
is edge-reconstructible.
\end{lemma}
\begin{proof} Without loss of generality assume that $Q_1,\dots,Q_k$ are ordered first by decreasing number
of vertices, then by decreasing number of edges. Then $S_{Q_1}^X (M) = S_{Q_1}(M).$ Now let $j > 1$ and assume that for each $i < j$ the value of $S_{Q_i}^X (M)$ has been computed. Then 
\[  S_{Q_j}^X (M) = S_{Q_j} (M) - \sum_{i=1}^{j-1} S_{Q_i}^X(M) S_{Q_j}(Q_i). \qedhere \] \end{proof} 

\begin{corollary} \label{appcor} \mbox{ } 
\begin{enumerate}
\item \label{c1} For a multigraph $Q$ with fewer edges than $M$, the parameter $S_Q^* (M )$ is edge-reconstructible, where $S_Q^* (M)$ counts the times $Q$ appears as an induced subgraph of $M$.
\item \label{c2} The multiset of multiplicities of the multiedges of $M$ is edge-reconstructible.
\item \label{c3} The multiset consisting of, for each vertex $v$ of $M$, the multiset of multiplicities of multiedges incident to $v$ is edge-reconstructible.
\item \label{c4} If $M$ is disconnected, then $M$ is edge-reconstructible.
\end{enumerate} 
\end{corollary}
\begin{proof}
(\ref{c1}) Suppose $Q$ is a multigraph with fewer edges than $M$. If $Q$ has the same number of vertices as $M$ then $S_Q^* (M) = 0$, so assume $Q$ has fewer vertices than $M$. Let $X$ be the set of all subgraphs of $M$ with the same number of vertices as $Q$. Then the lemma applies, and $S_Q^* (M) = S_Q^X(M).$

(\ref{c2}) We can calculate $S_Q^* (M)$ for any multiedge $Q$ if $M$ has at least $3$ vertices, and otherwise $M$
is a multiedge with multiplicity equal to the number of cards in the deck.

(\ref{c3}) Let $A$ be the desired multiset. If $G$ is a star then $M$ can be drawn by letting its multiedges all share a common vertex. Otherwise, let $X$ be the set of all subgraphs of $M$ whose underlying simple graphs are stars containing at least $3$ vertices. Then the lemma applies and accounts for all elements of $A$ except for the singletons containing the multiplicity of a multiedge reaching a leaf of $G$. If $e$ is an $m$-edge, and $m$ appears $k$ times total in the non-singleton multisets of $A$, then $\{m\}$ should appear $2S_e^*(M) - k$ times in $A$.

(\ref{c4}) Let $X$ be the set of all connected subgraphs of $M$. Then the lemma applies, and $M$ is the multigraph whose multiset of components contains precisely, for each $Q \in X$, $S_Q^X(M)$ copies of $Q.$
\end{proof}


\medskip 

{\footnotesize \noindent \emph{Appendix author address} \\
Wolfram Research \\
100 Trade Center Drive \\
Champaign, IL 61820-7237 \\
USA
{\tt  dmcdonald@wolfram.com} 
}

\bibliographystyle{amsplain}

\end{document}